\documentclass[reqno]{amsart}
\usepackage{amsfonts}

\newtheorem{theorem}{Theorem}
\theoremstyle{plain}

\newtheorem{lemma}{Lemma}

\newtheorem{proposition}{Proposition}

\numberwithin{equation}{section}
\numberwithin{theorem}{section}
\numberwithin{algorithm}{section}
\numberwithin{axiom}{section}
\numberwithin{case}{section}
\numberwithin{claim}{section}
\numberwithin{conclusion}{section}
\numberwithin{condition}{section}
\numberwithin{conjecture}{section}
\numberwithin{corollary}{section}
\numberwithin{criterion}{section}
\numberwithin{definition}{section}
\numberwithin{example}{section}
\numberwithin{exercise}{section}
\numberwithin{lemma}{section}
\numberwithin{notation}{section}
\numberwithin{problem}{section}
\numberwithin{proposition}{section}
\numberwithin{remark}{section}
\numberwithin{solution}{section}

\input{tcilatex}

\begin{document}
\title{Remarks on a mean field equation on $\mathbb{S}^{2}$}
\author{Changfeng Gui}
\address{Department of Mathematics, University of Texas at San Antonio, San
Antonio, TX 78249}
\email{changfeng.gui@utsa.edu}
\author{Fengbo Hang}
\address{Courant Institute, New York University, 251 Mercer Street, New
York, NY 10012}
\email{fengbo@cims.nyu.edu}
\author{Amir Moradifam}
\address{Department of Mathematics, University of California, Riverside, CA
92521}
\email{moradifam@math.ucr.edu}
\author{Xiaodong Wang}
\address{Department of Mathematics, Michigan State University, East Lansing,
MI 48824}
\email{xwang@math.msu.edu}
\dedicatory{Dedicated to Professor A. Chang and P. Yang on the occasion of
their 70th birthdays.}

\begin{abstract}
In this note, we study symmetry of solutions of the elliptic equation 
\begin{equation*}
-\Delta _{\mathbb{S}^{2}}u+3=e^{2u}\ \ \hbox{on}\ \ \mathbb{S}^{2},
\end{equation*}%
that arises in the study of rigidity problem of Hawking mass in general
relativity. We provide various conditions under which this equation has only
constant solutions, and consequently imply the rigidity of Hawking mass for
stable constant mean curvature (CMC) sphere.
\end{abstract}

\maketitle

\section{Introduction}

The main aim of this note is to study the semilinear elliptic equation 
\begin{equation}
-\Delta_{\mathbb{S}^{2}}u+\alpha=e^{2u}  \label{eqmain}
\end{equation}
on the standard $\mathbb{S}^{2}$. Here $u$ is a smooth function on $\mathbb{S%
}^{2}$ and $\alpha$ is a positive constant.

When $\alpha =1$, (\ref{eqmain}) means that the conformal metric $e^{2u}g_{%
\mathbb{S}^{2}}$ has constant curvature $1$. Therefore all solutions are
given by the pull back of the standard metric by Mobius transformations.
This and more general statements also follow from the powerful method of
moving plane (see \cite{ChL, GNN}). The latter approach can be used to show (%
\ref{eqmain}) has only constant solution when $0<\alpha <1$ (see \cite{Ln2}%
). More recently, the sphere covering inequality was discovered in \cite{GuM}
and applied to show all solutions to equation (\ref{eqmain}) must be
constant functions for $1<\alpha \leq 2$. In particular, this confirms a
long-standing conjecture of Chang-Yang (\cite{CY1, CY2}) concerning the best
constant in Moser-Trudinger type inequalities. Sphere covering inequality
and its generalization can also be used to solve many uniqueness and
symmetry problems from mathematical physics (see \cite{BGJM, GuM, GuHM} and
many references therein). \cite{GuHM} explains the sphere covering
inequality from the point view of comparison geometry and provides some
further generalizations. In contrast, for $2<\alpha <3$, nontrivial axially
symmetric solutions were found in \cite{Ln1}. The multiplicity of these
nontrivial axially symmetric solutions was carefully discussed in \cite{DET}%
. More recently, non-axially symmetric solutions to (\ref{eqmain}) for $%
\alpha >4$ but close to $4$ were found in \cite{GuHu}. In related
developments, topological degree of (\ref{eqmain}) for $\alpha \notin 
\mathbb{Z}$ was computed in \cite{ChLn, Ln1, L}. We refer the readers to the
survey article \cite{T} for more details of mean field equations on a closed
surface.

Recently, \cite{S, SSTW} discovered the interesting connection between the
equation (\ref{eqmain}) with $\alpha=3$ and rigidity problems involving
Hawking mass in general relativity. Among other results, it was shown in 
\cite{SSTW} that for $2<\alpha<4$, any even solution to (\ref{eqmain}) must
be axially symmetric. In particular, when $\alpha=3$, any even solution, $%
u(x)=u(-x)$ for all $x\in \mathbb{S}^{2}$, must be a constant function. It
is also conjectured in \cite[ section 3]{SSTW} that for $2<\alpha\leq3$, any
solution to (\ref{eqmain}) must be axially symmetric. Our note is motivated
by this conjecture. Our main result is

\begin{theorem}
\label{thm1.1}Assume $2<\alpha\leq3$ and $u\in C^{\infty}\left( \mathbb{S}%
^{2}\right) $ is a solution to%
\begin{equation*}
-\Delta_{\mathbb{S}^{2}}u+\alpha=e^{2u}.
\end{equation*}
If for some $p\in\mathbb{S}^{2}$, $\nabla u\left( p\right) =0$ and $%
D^{2}u\left( p\right) $ has two equal eigenvalues, then $u$ is axially
symmetric with respect to $p$. In particular, in the case $\alpha=3$, $u$
must be a constant function.
\end{theorem}

We may call the point $p$ in the assumption as an umbilical critical point
of $u$. So the theorem reads as: for $2<\alpha \leq 3$, any solution with an
umbilical critical point must be axially symmetric with respect to that
point. Here we do not know whether the solution is even or not. On the other
hand, the approach to Theorem \ref{thm1.1} can help us relax the even
assumption in \cite{SSTW} a little bit. One typical example is

\begin{theorem}
\label{thm1.2}Assume $2<\alpha\leq3$ and $u\in C^{\infty}\left( \mathbb{S}%
^{2}\right) $ is a solution to%
\begin{equation*}
-\Delta_{\mathbb{S}^{2}}u+\alpha=e^{2u}.
\end{equation*}
If every large circle splits $\mathbb{S}^{2}$ as two half sphere with equal
area under the metric $e^{2u}g_{\mathbb{S}^{2}}$, then $u$ is axially
symmetric with respect to some point. In particular, in the case $\alpha=3$, 
$u$ must be a constant function.
\end{theorem}

Note that if $u$ is even, then any large circle clearly splits the area. In
Section \ref{sec3}, we will also present several other conditions which is
weaker than the even assumption (see Proposition \ref{prop3.1}, \ref{prop3.2}%
). It is unfortunate we are not able to remove any of these assumptions.

At last we point out that there is an analogous nonlocal problem on $\mathbb{%
S}^{1}$, namely%
\begin{equation*}
\left\{ 
\begin{tabular}{l}
$\Delta u=0$ on $B_{1}^{2}\subset \mathbb{R}^{2},$ \\ 
$\frac{\partial u}{\partial \nu }+\lambda =e^{u}$ on $\mathbb{S}^{1}.$%
\end{tabular}%
\ \right.
\end{equation*}%
Here $\nu $ is the unit outer normal direction and $\lambda $ is a positive
constant. This equation appears in the study of determinant of Laplacian on
compact surface with boundary (see \cite{OPS}). The solutions to the above
problem is well understood (see \cite[Lemma 2.3]{OPS} and \cite[Theorem 3]{W}%
). The reason the problem on $\mathbb{S}^{1}$ is much simpler than (\ref%
{eqmain}) is because the Fourier analysis on $\mathbb{S}^{1}$ is much easier.

In Section \ref{sec2}, we will describe our main new observation and use it
to derive Theorem \ref{thm1.1}. In Section \ref{sec3}, we will apply this
new observation to derive several relaxation of the even assumption in \cite%
{SSTW}. In particular, Theorem \ref{thm1.2} will be proved.

\section{Proof of Theorem \protect\ref{thm1.1}\label{sec2}}

Let us fix the notations.%
\begin{equation}
\mathbb{S}^{2}=\left\{ x\in \mathbb{R}^{3}:x_{1}^{2}+x_{2}^{2}+x_{3}^{2}=1%
\right\} .  \label{eq2.1}
\end{equation}%
$e_{1}=\left( 1,0,0\right) ,e_{2}=\left( 0,1,0\right) $ and $e_{3}=\left(
0,0,1\right) $ denotes the standard base. For $x,y\in \mathbb{R}^{3}$, $%
x\cdot y$ denotes the usual dot product. For any $y\in \mathbb{S}^{2}$, we
write%
\begin{equation}
H_{y}=\left\{ x\in \mathbb{S}^{2}:x\cdot y\geq 0\right\}  \label{eq2.2}
\end{equation}%
as the half sphere, and%
\begin{equation}
C_{y}=\left\{ x\in \mathbb{S}^{2}:x\cdot y=0\right\}  \label{eq2.3}
\end{equation}%
as the large circle. We also denote $R_{y}$ as the reflection with respect
to plane $\left\{ x\in \mathbb{R}^{3}:x\cdot y=0\right\} $ i.e.%
\begin{equation}
R_{y}x=x-2\left( x\cdot y\right) y.  \label{eq2.4}
\end{equation}

Assume $u\in C^{\infty }\left( \mathbb{S}^{2},\mathbb{R}\right) $ satisfies (%
\ref{eqmain}) with $2<\alpha \leq 3$, then%
\begin{equation}
\int_{\mathbb{S}^{2}}e^{2u}d\mu =4\pi \alpha .  \label{eq2.5}
\end{equation}%
Here $\mu $ is the measure associated with standard metric.

Fix $y\in \mathbb{S}^{2}$, we define $v\left( x\right) =u\left(
R_{y}x\right) $ for any $x\in \mathbb{S}^{2}$ and $w=u-v$. Then%
\begin{equation}
\left\{ 
\begin{array}{l}
-\Delta _{\mathbb{S}^{2}}w=e^{2u}-e^{2v}=cw; \\ 
\left. w\right\vert _{C_{y}}=0.%
\end{array}%
\right.  \label{eq2.6}
\end{equation}%
Here%
\begin{equation}
c\left( x\right) =2\int_{0}^{1}e^{2\left( \left( 1-t\right) v\left( x\right)
+tu\left( x\right) \right) }dt  \label{eq2.7}
\end{equation}%
is a smooth function on $\mathbb{S}^{2}$.

If $w$ is not identically zero, then classical results (see \cite{Be, Chg,
HW}) imply that the nodal set of $w$ consists of finitely many smooth curves
which only intersects at critical points of $w$. Moreover $w$ behaves like a
harmonic polynomial near each critical point i.e. nodal set locally looks
like straight lines with equal angles at critical points.

If $\Omega \subset H_{y}$ is a simply connected nodal domain, then it
follows from the sphere covering inequality (\cite[Theorem 1.1]{GuM}), or
more precisely, the formulation with standard $\mathbb{S}^{2}$ as background
metric (\cite[Proposition 3.1]{GuHM}), that%
\begin{equation}
\int_{\Omega \cup R_{y}\left( \Omega \right) }e^{2u}d\mu =\int_{\Omega
}e^{2u}d\mu +\int_{\Omega }e^{2v}d\mu >4\pi .  \label{eq2.8}
\end{equation}%
This inequality and (\ref{eq2.5}) implies $H_{y}$ can not contain $3$ or
more simply connected nodal domains.

The crucial step to prove symmetry of solutions in \cite{GuM, SSTW} is
counting the number of simply connected nodal domains. As observed in \cite[%
section 4.2]{GuM}, if we have a critical point of $u$ on $C_{y}$, namely $%
q\in C_{y}$, and $w$ is not identically zero, then the order of $w$ at $q$
(i.e. the order of the first nonvanishing term in Taylor expansion of $w$ at 
$q$) is at least $2$. Hence in $H_{y}$, at least one nodal line emanates
from $q$ with equal angle. This implies $H_{y}$ contains at least two simply
connected nodal domains.

Let $z$ be the unit tangent vector of $C_{y}$ at $q$. Our new observation
is: if $z$ is an eigenvector of $D^{2}u\left( q\right) $, and $w$ is not
identically zero, then the order of $w$ at $q$ is at least $3$. If the order
is larger than or equal to $4$, then $H_{y}$ contains at least $3$ simply
connected nodal domains, and it contradicts with (\ref{eq2.5}). When the
order of $w$ at $q$ is $3$, the nodal set of $w$ emanates two lines from $q$
with angle $\frac{\pi }{3}$ in between, and $w$ takes alternating signs in
each angle. Since we can not have $3$ or more simply connected nodal
domains, the only possibility is we have only two nodal domains (i.e. the
two emanating nodal line from $q$ form a loop in $H_{y}$). It follows from
Hopf principle that $\frac{\partial w}{\partial \nu }$ is nonzero and of a
fixed sign on $C_{y}\backslash \left\{ q\right\} $, here $\nu $ is the unit
outer normal vector of $H_{y}$ (in fact, $\nu =-y$). In particular, there is
no critical point on $C_{y}\backslash \left\{ q\right\} $. We state this
conclusion as a lemma.

\begin{lemma}
\label{lem2.1}Assume $u$ is a smooth solution to (\ref{eqmain}) with $%
2<\alpha \leq 3$. Let $q$ be a critical point of $u$ and $z\in T_{q}\mathbb{S%
}^{2}$ be an unit eigenvector of $D^{2}u\left( q\right) $. Denote $y=q\times
z$ (the cross product in $\mathbb{R}^{3}$), $v=u\circ R_{y}$ and $w=u-v$. If 
$w$ is not identically zero, then $\frac{\partial w}{\partial \nu }$ is
nonzero and of a fixed sign on $C_{y}\backslash \left\{ q\right\} $, here $%
\nu =-y$ is the unit outer normal vector of $H_{y}$. Note that on $C_{y}$, $%
\frac{\partial w}{\partial \nu }=2\frac{\partial u}{\partial \nu }$. Hence $%
\frac{\partial u}{\partial \nu }$ is nonzero and of a fixed sign on $%
C_{y}\backslash \left\{ q\right\} $.
\end{lemma}

With Lemma \ref{lem2.1} at hand, we are ready to derive Theorem \ref{thm1.1}.

\begin{proof}[Proof of Theorem \protect\ref{thm1.1}]
By rotation we can assume the umbilical critical point $p=e_{3}$. It follows
from \cite[Lemma 5]{SSTW} that%
\begin{equation}
\int_{S^{2}}x_{i}u\left( x\right) d\mu \left( x\right) =0  \label{eq2.9}
\end{equation}%
for $i=1,2,3$. Since every nonzero vector in $T_{e_{3}}\mathbb{S}^{2}$ is an
eigenvector of $D^{2}u\left( e_{3}\right) $, we know for every $y\in
C_{e_{3}}$, we can apply Lemma \ref{lem2.1} with $q=e_{3}$.

If $\nabla u\left( -e_{3}\right) =0$, then for any $y\in C_{e_{3}}$, $u\circ
R_{y}=u$ (i.e. $u$ is symmetric with respect to $C_{y}$). Hence $u$ is
axially symmetric with respect to $e_{3}$.

Next we claim $\nabla u\left( -e_{3}\right) $ must be zero. In fact, if this
is not the case. Since $u$ must have another critical point besides $e_{3}$,
by rotation we can assume it lies in $C_{e_{2}}\backslash \left\{
e_{3}\right\} $. It follows from Lemma \ref{lem2.1} that $u$ is symmetric
with respect to $C_{e_{2}}$. For $0<\theta <\pi $, we know $u$ can not be
symmetric with respect to $C_{-\sin \theta \cdot e_{1}+\cos \theta \cdot
e_{2}}$ (because otherwise we conclude 
\begin{equation*}
\nabla u\left( -e_{3}\right) \cdot e_{2}=0,
\end{equation*}%
and%
\begin{equation*}
\nabla u\left( -e_{3}\right) \cdot \left( -\sin \theta \cdot e_{1}+\cos
\theta \cdot e_{2}\right) =0,
\end{equation*}%
it follows that $\nabla u\left( -e_{3}\right) =0$, a contradiction). Hence
we know $\frac{\partial u}{\partial \nu }$ does not vanish on $C_{-\sin
\theta \cdot e_{1}+\cos \theta \cdot e_{2}}\backslash \left\{ e_{3}\right\} $
and it has a fixed sign. Here $\nu =\sin \theta \cdot e_{1}-\cos \theta
\cdot e_{2}$. By continuity we know the sign also does not depend on $\theta 
$. It follows that for $-1\leq s<1$, $u\left( \sqrt{1-s^{2}}\cos \theta ,%
\sqrt{1-s^{2}}\sin \theta ,s\right) $ is strictly monotone in $\theta \in
\left( 0,\pi \right) $. This contradicts with the fact%
\begin{equation*}
\int_{S^{2}}x_{1}u\left( x\right) d\mu \left( x\right) =0.
\end{equation*}%
Hence $\nabla u\left( -e_{3}\right) $ must be zero.

If $\alpha =3$, it follows from the fact $u$ is axially symmetric and \cite[%
Proposition 1]{SSTW} that $u$ must be a constant function.
\end{proof}

\section{Some relaxation of even assumption\label{sec3}}

Here we want to show the discussion in Section \ref{sec2} can help us relax
the even assumption in \cite{SSTW}.

\begin{proof}[Proof of Theorem \protect\ref{thm1.2}]
Note that the equal area assumption can be written as: for any $y\in \mathbb{%
S}^{2}$,%
\begin{equation}
\int_{H_{y}}e^{2u}d\mu =\int_{H_{-y}}e^{2u}d\mu .  \label{eq3.1}
\end{equation}%
Assume $q$ is a critical point of $u$ and $z\in T_{q}\mathbb{S}^{2}$ is an
eigenvector of $D^{2}u\left( q\right) $. Denote $y=q\times z$, $v=u\circ
R_{y}$ and $w=u-v$. Then $w$ must be identically zero i.e. $u$ is symmetric
with respect to $C_{y}$. Indeed if $w$ is not identically zero, it follows
from Lemma \ref{lem2.1} that $\frac{\partial w}{\partial \nu }$ is nonzero
and of a fixed sign on $C_{y}\backslash \left\{ q\right\} $, here $\nu =-y$
is the unit outer normal vector of $H_{y}$. Using%
\begin{equation*}
-\Delta _{\mathbb{S}^{2}}w=e^{2u}-e^{2v},
\end{equation*}%
we see%
\begin{eqnarray*}
&&\int_{H_{y}}e^{2u}d\mu -\int_{H_{-y}}e^{2u}d\mu \\
&=&\int_{H_{y}}e^{2u}d\mu -\int_{H_{y}}e^{2v}d\mu \\
&=&-\int_{H_{y}}\Delta _{\mathbb{S}^{2}}wd\mu \\
&=&-\int_{C_{y}}\frac{\partial w}{\partial \nu }ds \\
&\neq &0.
\end{eqnarray*}%
This contradicts with the equal area assumption.

By rotation we can assume $e_{3}$ is a critical point of $u$, and $%
D^{2}u\left( e_{3}\right) $ has $e_{1},e_{2}$ as eigenvectors. It follows
from previous discussion that $u$ is symmetric with respect to $C_{e_{1}}$
and $C_{e_{2}}$.

We will show $u$ must be symmetric with respect to $C_{e_{3}}$. One this is
known it follows from \cite[Lemma 8]{SSTW} that $u$ must be axially
symmetric.

To continue we let $v=u\circ R_{e_{3}}$ and $w=u-v$, then using the equal
area assumption, same argument as above shows%
\begin{equation*}
\int_{C_{e_{3}}}\frac{\partial w}{\partial e_{3}}ds=0=2\int_{C_{e_{3}}}\frac{%
\partial u}{\partial e_{3}}ds.
\end{equation*}%
Hence $\frac{\partial u}{\partial e_{3}}$ vanishes at some point on $%
C_{e_{3}}$.

If $\frac{\partial u}{\partial e_{3}}\left( e_{1}\right) =0$, then it
follows from the fact $u$ is symmetric with respect to $C_{e_{2}}$ that $%
\nabla u\left( e_{1}\right) =0$ and $D^{2}u\left( e_{1}\right) $ has $%
e_{2},e_{3}$ as eigenvectors. On the other hand, it follows from the fact $u$
is symmetric with respect to $C_{e_{1}}$ that $\nabla u\left( -e_{1}\right)
=0$. This together with Lemma \ref{lem2.1} implies $u-u\circ R_{e_{3}}=0$
i.e. $u$ is symmetric with respect to $C_{e_{3}}$.

If $\frac{\partial u}{\partial e_{3}}\left( e_{2}\right) =0$, we can get
symmetry with respect to $C_{e_{3}}$ exactly as above.

If $\frac{\partial u}{\partial e_{3}}\left( q\right) =0$ for some $q\in
C_{e_{3}}\backslash \left\{ \pm e_{1},\pm e_{2}\right\} $, we must have $%
w=u-u\circ R_{e_{3}}$ is identically zero. If it is not the case, then the
order of $w$ at $q$ is at least $2$. By symmetry, the order of $w$ at $%
R_{e_{1}}q$, $R_{e_{2}}q$ and $-q$ must be at least $2$ too. Hence $w$ has
at least $3$ simply connected nodal domain in $H_{e_{3}}$. This contradicts
with the sphere covering inequality by the discussion in Section \ref{sec2}.

In all the cases, we know $u$ must be axially symmetric, and hence it must
be constant when $\alpha =3$ (\cite[Proposition 1]{SSTW}).
\end{proof}

Along the same line we have the following

\begin{proposition}
\label{prop3.1}Assume $2<\alpha \leq 3$ and $u\in C^{\infty }\left( \mathbb{S%
}^{2}\right) $ is a solution to%
\begin{equation*}
-\Delta _{\mathbb{S}^{2}}u+\alpha =e^{2u}.
\end{equation*}%
If there exists $p\in \mathbb{S}^{2}$ such that both $p$ and $-p$ are
critical points of $u$, and%
\begin{equation*}
\int_{H_{p}}e^{2u}d\mu =\int_{H_{-p}}e^{2u}d\mu ,
\end{equation*}%
then $u$ must be axially symmetric. If $\alpha =3$, then $u$ is a constant
function.
\end{proposition}

\begin{proof}
By rotation we can assume $p=e_{3}$ and $D^{2}u\left( e_{3}\right) $ has $%
e_{1},e_{2}$ as eigenvector. It follows from Lemma \ref{lem2.1} and the fact 
$\nabla u\left( -e_{3}\right) =0$ that $u$ is symmetric with respect to $%
C_{e_{1}}$ and $C_{e_{2}}$. Now using%
\begin{equation*}
\int_{H_{e_{3}}}e^{2u}d\mu =\int_{H_{-e_{3}}}e^{2u}d\mu ,
\end{equation*}%
the argument in the proof of Theorem \ref{thm1.2} tells us $u$ is also
symmetric with respect to $C_{e_{3}}$. It follows from \cite[Lemma 8]{SSTW}
that $u$ must be axially symmetric.
\end{proof}

\begin{proposition}
\label{prop3.2}Assume $2<\alpha \leq 3$ and $u\in C^{\infty }\left( \mathbb{S%
}^{2}\right) $ is a solution to%
\begin{equation*}
-\Delta _{\mathbb{S}^{2}}u+\alpha =e^{2u}.
\end{equation*}%
If there exists $p\in \mathbb{S}^{2}$ such that $\nabla u\left( p\right) =0$%
, $\nabla u\left( -p\right) =0$ and $D^{2}u\left( p\right) =D^{2}u\left(
-p\right) $ (here we identify $T_{p}S^{2}$ with $T_{-p}S^{2}$ naturally),
then $u$ must be axially symmetric. If $\alpha =3$, then $u$ is a constant
function.
\end{proposition}

\begin{proof}
By rotation we can assume $p=e_{3}$ and $D^{2}u\left( e_{3}\right) $ has $%
e_{1},e_{2}$ as eigenvector. It follows from Lemma \ref{lem2.1} and the fact 
$\nabla u\left( -e_{3}\right) =0$ that $u$ is symmetric with respect to $%
C_{e_{1}}$ and $C_{e_{2}}$. It follows from the equation that%
\begin{equation*}
e^{2u\left( e_{3}\right) }=-\Delta _{\mathbb{S}^{2}}u\left( e_{3}\right)
+\alpha =-\Delta _{\mathbb{S}^{2}}u\left( -e_{3}\right) +\alpha =e^{2u\left(
-e_{3}\right) }.
\end{equation*}%
Hence $u\left( e_{3}\right) =u\left( -e_{3}\right) $. Let $w=u-u\circ
R_{e_{3}}$, then because $w$ is symmetric with respect to $C_{e_{1}}$ and $%
C_{e_{2}}$, we see $w$ vanishes at least to order $4$ (does not include $4$)
at $e_{3}$. If $w$ is not identically zero, then it will have at least $3$
simply connected nodal domains. This contradicts with the sphere covering
inequality by the discussion in Section \ref{sec2}. It follows from \cite[%
Lemma 8]{SSTW} that $u$ must be axially symmetric.
\end{proof}

\end{document}